\documentclass[reqno]{article}
\usepackage{lineno,hyperref}
\usepackage{amsmath,amssymb}
\usepackage{amssymb}
\modulolinenumbers[5]
\usepackage{geometry}
\newtheorem{Theorem}{Theorem}[section]

\newtheorem{Lemma}[Theorem]{Lemma}
\newtheorem{Corollary}[Theorem]{Corollary}
\newtheorem{Remark}[Theorem]{Remark}

\newcommand{\cvd}{\hfill$\square$ \bigskip}

\numberwithin{equation}{section}

\begin{document}

%
%
%
%
%
%
%
%
%

\title {Improved lower bounds for Dirichlet eigenvalues of the Laplacian and poly-Laplacian on bounded Euclidean domains}

\author{\sc Zhengchao Ji \footnote{
Department of Mathematics, China Jiliang University,
             Hangzhou,  310027, People's Republic of China. E-mail: jizhengchao@zju.edu.cn} and Yong Luo  \footnote{
Mathematical Science Research Center, Chongqing University of Technology,
             Chongqing,   400054, People's Republic of China. E-mail: yongluo-math@cqut.edu.cn}}

\date{}
\maketitle




\begin{abstract}

In this paper, we establish Brezin-Li-Yau type lower bounds for averaged sums of Dirichlet eigenvalues of the Laplacian and poly-Laplacian on bounded domains in Euclidean spaces. By deriving expansions of two binary polynomials which may be of independent interest, we improve several existing lower bounds of this kind in the literature. Furthermore, our lower bounds are optimal in the sense that our expansions capture all positive terms, whereas previous works only provided certain lower bounds for these two binary polynomials, effectively capturing only a subset of the positive terms identified in our expansions.
\end{abstract}
\vspace{0.2cm}

\section{Introduction}\hspace{5 mm}  \setcounter{equation}{0}
Let $\Omega$ be a bounded domain with piecewise smooth boundary $\partial \Omega$ in an $n$-dimensional
Euclidean space $\mathbb{R}^n$. First of all, we focus on the following  Dirichlet eigenvalue problem of the Laplacian
\begin{equation}\label{DLE}
\begin{cases}
\Delta u=-\lambda u  \ & \mathrm{in}\ \Omega, \\
u=0  \ &  \mathrm{on}\ \partial\Omega.
\end{cases}
\end{equation}
It is well known that spectrum of the eigenvalue problem (\ref{DLE}) is real and discrete (cf. \cite{C,J,PPW})
\begin{alignat*}{1}
0<\lambda_1<\lambda_2\leq \lambda_3\leq \cdots\rightarrow\infty,
\end{alignat*}
where each eigenvalue is repeated according to its multiplicity.

Let $V(\Omega)$ (or $V$) be the volume of $\Omega$,  and   $\omega_n$   the volume of the unit ball in $\mathbb{R}^n$.
Then the following well-known Weyl's asymptotic formula holds
\begin{alignat*}{1}
\lambda_k\sim \frac{4\pi^2}{(\omega_nV(\Omega))^\frac{2}{n}}k^{\frac{2}{n}},\ k\rightarrow\infty,
\end{alignat*}
which implies that
\begin{alignat}{1}\label{EAF}
\frac{1}{k}\sum_{i=1}^k \lambda_i \sim \frac{n}{n+2}\frac{4\pi^2}{(\omega_nV(\Omega))^\frac{2}{n}}k^{\frac{2}{n}},\,\,\,\,\, k\rightarrow\infty.
\end{alignat}
In 1961, P\'{o}lya \cite{Po} proved that, if $n=2$ and $\Omega$ is a \textit{tiling domain} in $\mathbb{R}^2$, then
\begin{alignat*}{1}
\lambda_k \geq \frac{4\pi^2}{(\omega_nV(\Omega))^\frac{2}{n}}k^{\frac{2}{n}},\ \ \mathrm{for}\ k=1,2,\ldots.
\end{alignat*}
Based on the result above, he proposed the famous conjecture:\\

\textbf{P\'{o}lya's conjecture} \textit{If $\Omega$ is a bounded domain in $\mathbb{R}^n$, then the $k$-th eigenvalue $\lambda_k$ of the eigenvalue
problem (\ref{DLE}) satisfies}
\begin{alignat*}{1}
\lambda_k \geq \frac{4\pi^2}{(\omega_nV(\Omega))^\frac{2}{n}}k^{\frac{2}{n}},\ \ \mathrm{for}\ k=1,2,\ldots.
\end{alignat*}

During the past six decades, many mathematicians have focused on this problem and related  topics, and there are a lot of important results on this subject (cf. \cite{Be,BE,CW1,CW2,FMS,GJL,HW, KVW,La,Li,L}). We suggest  that   readers refer \cite{SY,Y} for  more details. In 1983, Li and Yau \cite{LY} verified the famous   inequality
\begin{alignat}{1}\label{LYE}
\frac{1}{k}\sum_{i=1}^k \lambda_i \geq \frac{n}{n+2}\frac{4\pi^2}{(\omega_nV(\Omega))^\frac{2}{n}}k^{\frac{2}{n}},\ k=1,2,\ldots.
\end{alignat}
 By using Legendre transformation of an earlie result obtained by Berezin \cite{Be}, one can also obtain the above inequality. Nowadays the above inequality is called the Brezin-Li-Yau's inequality in the literature.

From (\ref{LYE}), one can derive that
\begin{alignat*}{1}
\lambda_k \geq   \frac{n}{n+2}\frac{4\pi^2}{(\omega_nV(\Omega))^\frac{2}{n}}k^{\frac{2}{n}},\ \ \mathrm{for}\ k=1,2,\ldots,
\end{alignat*}
which gives a partial solution to the P\'{o}lya's conjecture  with a factor $\frac{n}{n+2}$. Recently, Filonov, Levitin, Polterovich and Sher \cite{FLPS,FLPS2} made some breakthroughs by solving the P\'{o}lya's Conjecture
for balls in $\mathbb{R}^n$ and annuli in $\mathbb{R}^2$.

In \cite{M}, Melas obtained the following beautiful estimate which improves (\ref{LYE}) for $n\geq 2$ and $k\geq 1$
\begin{alignat}{1}\label{MLE}
\frac{1}{k}\sum_{i=1}^k \lambda_i \geq   \frac{n}{n+2}\frac{4\pi^2}{(\omega_nV(\Omega))^\frac{2}{n}}k^{\frac{2}{n}}+c_n\frac{V(\Omega)}{I(\Omega)},\ \ \mathrm{for}\ k=1,2,\ldots,
\end{alignat}
where $c_n\leq 1/(24(n+2))$ is a positive constant depending only on  $n$ and
\begin{alignat*}{1}
I(\Omega)=\min_{a\in  \mathbb{R}^n} \int_{\Omega} |x-a|^2dx
\end{alignat*}
is called the moment of \textit{inertia} of $\Omega$.

Due to \cite{Iv1,Iv2,Me,Sa}, under suitable conditions on $\Omega$, we have the following asymptotic formula
\begin{alignat*}{1}
\frac{1}{k}\sum_{i=1}^k\lambda_i\sim\frac{n}{n+2}\frac{4\pi^2 k^{\frac{2}{n}}}{(\omega_nV(\Omega))^\frac{2}{n}}
+\frac{\sqrt{\pi}\Gamma\left(\frac{4+n}{2}\right)^{1+\frac{1}{n}}|\partial\Omega|k^{\frac{1}{n}}}{(n+1)\Gamma
\left(\frac{3+n}{2}\right)\Gamma(2)^{\frac{1}{n}}V(\Omega)^{1+\frac{1}{n}}}
,\,\,\,\,k\rightarrow\infty.
\end{alignat*}
Hence, it is natural to find the following type lower bound
\begin{alignat*}{1}
\frac{1}{k}\sum_{i=1}^k\lambda_i\geq \frac{n}{n+2}\frac{4\pi^2 k^{\frac{2}{n}}}{(\omega_nV(\Omega))^\frac{2}{n}}
+\frac{\sqrt{\pi}\Gamma\left(\frac{4+n}{2}\right)^{1+\frac{1}{n}}|\partial\Omega|k^{\frac{1}{n}}}{(n+1)\Gamma\left(\frac{3+n}{2}\right)\Gamma(2)^{\frac{1}{n}}V(\Omega)^{1+\frac{1}{n}}}
,\,\,\,\,k=1,2,\cdots.
\end{alignat*}

In 2009, Kova\v{r}\'{i}k, Vugalter and Weidl \cite{KVW} first obtained this kind of results on polygons and established that
\begin{alignat}{1}\label{KVW}
\frac{1}{k}\sum_{i=1}^k \lambda_i \geq  \frac{2\pi}{V(\Omega)}k+4\bar{c}k^{\frac{1}{2}-\varepsilon(k)}V(\Omega)^{-\frac{3}{2}}\sum^m_{j=1}l_j\Theta\left(k-\frac{9V}{2\pi d^2_j}\right),
\end{alignat}
where $m,l_j,d_j,\Theta(x)$ are some constants depending on $\partial\Omega$ and
\begin{alignat*}{1}
\varepsilon(k)=\frac{2}{\sqrt{\log_2(\frac{2\pi k}{c})}},\,\,\, c=\sqrt{\frac{3\pi}{14}}10^{-11},\,\,\, \bar{c}=\frac{2^{-3}}{9\sqrt{2}\cdot36}(2\pi)^{\frac{5}{4}}c^{\frac{1}{4}}.
\end{alignat*}
Compared with the second term on the right hand of (\ref{MLE}) with that of  (\ref{KVW}), Kova\v{r}\'{i}k-Vugalter-Weidl's lower bound is sharper for large $k$.

In 2010, Ilyin \cite{Il} improved Melas's lower bound when $n=2,3,4$ by proving that
\begin{alignat}{1}\label{Ilsl2}
\frac{1}{k}\sum_{i=1}^k\lambda_i\geq \frac{n}{n+2}\frac{4\pi^2}{(\omega_nV(\Omega))^\frac{2}{n}}k^{\frac{2}{n}} +\beta_n^L\frac{n}{48}\frac{V(\Omega)}{I(\Omega)},
\end{alignat}
where $\beta_n^L= 119/120,\beta_n^L=0.986,\beta_n^L=0.978$ for $n=2,3,4$ respectively. In 2013,  Yildirim and Yolcu  \cite{YY2} improved (\ref{MLE}) by adding the following  term to the right hand of  (\ref{MLE})
\begin{alignat}{1}\label{YYML}
\begin{split}
\frac{V(\Omega)^{\frac{3n+2}{2n}}}{144(n+2)I^{\frac{3}{2}}\Gamma^{\frac{1}{n}}(1+n/2)}k^{-\frac{1}{n}}.
\end{split}
\end{alignat}

In 2020, the first author  and Xu \cite{JX} proved a sharper polynomial inequality and  improved (\ref{YYML}) for any bounded domain $\Omega\subseteq \mathbb{R}^n$, $n\geq m+1\geq 3$ and  $ k \geq 1$ as
\begin{alignat}{1}\label{JX1}
\begin{split}
\frac{1}{k}\sum_{i=1}^k \lambda_i \geq & \omega_n^{-\frac{2}{n}} {\alpha^{-\frac{2}{n}}}k^{\frac{2}{n}}-\frac{2\omega_n^{\frac{m-1}{n}}S_{m+2}\alpha^{\frac{(m+1)n+m-1}{n}}}{(n+2)\rho^{m+1}}k^{\frac{-m+1}{n}}\\
&+c_2\frac{2\omega_n^{\frac{m}{n}}(m+1)S_{m+3}\alpha^{\frac{(m+2)n+m}{n}}}{(n+2)(m+3)\rho^{m+2}}k^{\frac{-m}{n}},
\end{split}
\end{alignat}
where $S_{l}=(a+1)^l-a^l,$
\begin{alignat*}{1}
\begin{split}
c_2\leq & \min \left\{1,  \frac{(m+1)n+m-1 }{(m+2)n+m}\frac{\sqrt{2}S_{m+2}}{S_{m+3}}\frac{m+3}{m+1}k^{\frac{1}{n}} \right\},
\end{split}
\end{alignat*}
\begin{alignat}{1}
\begin{split}\label{two quan.}
\alpha=& {V(\Omega)}/{(2\pi)^n},\,\,\,\rho=2(2\pi)^{-n}\sqrt{V(\Omega)I(\Omega)},
\end{split}
\end{alignat}
and $a$ is  defined by (\ref{DA1}). In 2025, Ji and Xu \cite{JX2} improved and generalized Ilyin's eigenvalue estimates for any $n\geq 2$ under
certain restrictive conditions.

For any positive integer $l\geq 1$, the Dirichlet eigenvalue  problem of the poly-Laplacian is defined by the following equation
\begin{equation}\label{EOE11}
\begin{cases}
(-\Delta)^l \bar{u}_j=\Lambda _j\bar{u}_j  \ & \mathrm{in}\ \Omega, \\
\bar{u}_j=\frac{\partial \bar{u}_j}{\partial \nu}=\cdots=\frac{\partial^{l-1} \bar{u}_j}{\partial \nu^{l-1}} = 0  \  & \mathrm{on}\ \partial\Omega.
\end{cases}
\end{equation}
If $l=2$,  equation (\ref{EOE11}) is called  the \textbf{clamped plate problem} \cite{CW1,CW2,J2,JX,JX2}, which governs motion of the vibration of a stiff plate. An important eigenvalue problem for the lower bound  of the clamped plate problem is the famous \textbf{Lord Rayleigh's conjecture} \cite{K1,K2}.
For the Brezin-Li-Yau type inequalities of the clamped plate problem, Cheng and Wei made important contributions \cite{CW1,CW2}, and Ji and Xu \cite{JX} obtained an important result  similar to (\ref{JX1}).

In \cite{KKT}, Ku-Ku-Tan proposed \textbf{ the generalized P\'{o}lya's Conjecture}  that the Dirichlet eigenvalues $\Lambda_k$ of the poly-Laplacian satisfy inequalities
\begin{alignat*}{1}
\Lambda_k
\geq &\omega^{\frac{2l}{n}}_n\alpha^{-\frac{2l}{n}} k^{\frac{2l}{n}}.
\end{alignat*}
Recently, Ji and Xu \cite{JX2} established a series of inequalities concerning this conjecture.

Before we state our main results, we would like to  discuss the following ways to obtain lower bounds for Dirichlet eigenvalues of the Laplacian and explain our methodology which helps us to obtain the improvements.

\textbf{{Li-Yau's way:}} Li-Yau \cite{LY} originally used the Fourier transform to covert the estimation of $\sum_{i=1}^k\lambda_i$ to the upper bound $M_1$ of certain decreasing radial rearrangement function $F_1^{*}(|\xi|)$. They proved the following key inequality
\begin{alignat*}{1}
\frac{n}{n+2}\left(\frac{n}{M_1\omega_{n-1}} \right)^\frac{2}{n}\left(\int_{\mathbb{R}^n} F^*_1(|\xi|)d\xi \right)^{\frac{n+2}{n}}\leq M_2,
\end{alignat*}
where $M_1$ is the upper bound of $F^*_1$ and $M_2$ is the upper bound of $\int_{\mathbb{R}^n}|\xi|^2F^*_1(|\xi|)d\xi$. Using properties of the Fourier series and the rearrangement inequality, Li-Yau proved the upper bound of  $M_1$ is  $\alpha$, then they obtained (\ref{LYE}).

\textbf{{Melas's way:}} Melas \cite{M} noticed that
\begin{alignat*}{1}
\bar{M}_1:=\int_{\mathbb{R}^n}|\xi|^2F^*_1(|\xi|)d\xi\leq \sum_{i=1}^k\lambda_i.
\end{alignat*}
Moreover, he proved  that the upper bound for the derivative of $F^*_1(|\xi|)$ is $\rho$.  Then, Melas established a lower bound for \textbf{ a three-terms binary polynomial} to obtain the lower bound for $\bar{M}_1$ and proved (\ref{MLE}). Later, Ji and Xu \cite{JX} found a better lower bound for this binary polynomial to control $\bar{M}_1$, then they improved (\ref{MLE}) by giving (\ref{JX1}).

Consequently, if one can provide a sharper lower bound  for $\bar{M}_1$ or a sharper upper bound for $M_1$,  the lower bound for $\sum_{i=1}^k\lambda_i$ can be improved. These methods can be also used to investigate other eigenvalue problems \cite{CW1,CQW,CSWZ,CW2,JX,J2,YY1,YY2}.

In the Melas's improvement of the Berezin-Li-Yau's inequality and further improvements of Melas's inequality (cf. \cite{JX}), the key step is to give a lower bound for the following binary polynomial
$$P(s,\tau):=ns^{n+2}-(n+2)\tau^2s^n+2\tau^{n+2}.$$
Along this line, we get the best possible estimate by proving an expansion for $P(s,\tau)$ (cf. Lemma \ref{ke1}). By using this expansion, we get a sharper lower bound in  Theorem \ref{C1} which improves (\ref{JX1}).

For the poly-Laplacian, the critical step of Ji-Xu's (cf. \cite{JX2}) improvements  is to obtian a better lower bound for the
binary polynomial
$$P_1(s,\tau,d):=ns^{n+l}-(n+l)\tau^ls^n+l\tau^{n+l}.$$
In section 3, we obtain an expansion for $P_1(s,\tau,d)$ (cf. Lemma \ref{KL22}). By using this expansion, we  prove a shaper inequality in Theorem \ref{C12} which is the best possible along this line, sharpening several results in \cite{JX,JX2}.

\section{Lower bounds for  Dirichlet eigenvalues of the  Laplacian}

In this section, we will prove our first main result.

\begin{Theorem}\label{C1}
For any bounded domain $\Omega\subseteq \mathbb{R}^n$, $n\geq 2$,  we have
\begin{alignat*}{1}
\begin{split}
\frac{1}{k}\sum_{i=1}^k \lambda_i
\geq & \frac{4\pi^2}{(\omega_nV(\Omega))^\frac{2}{n}}k^{\frac{2}{n}}-\frac{S_{n+1}}{(n+1)\rho^n}\omega^{\frac{n-2}{n}}_n\alpha^{\frac{(n+2)(n-1)}{n}}k^{\frac{2-n}{n}}\\
&+c_1\frac{nS_{n+3}}{(n+2)(n+3)\rho^{n+2}}\omega_n\alpha^{n+3}k^{-1},
\end{split}
\end{alignat*}
where  $S_{i}=(a+1)^i-a^i$, $a$ is defined by (\ref{DA1}), $\alpha, \rho$ are defined in (\ref{two quan.}) and
\begin{alignat*}{1}
 c_1=\min\left\{1, \frac{2^{n+3}(n+2)}{n^2 S_{n+3}}  k^{\frac{n+2}{n}}\right\}.
\end{alignat*}
\end{Theorem}

Firstly, we  introduce some notations and definitions. For a bounded domain $\Omega$, \textit{the moment of inertia} of $\Omega$ is defined by
\begin{alignat*}{1}
I(\Omega)=\min_{a\in \mathbb{R}^n}\int_{\Omega} |x-a|^2 dx.
\end{alignat*}
By a translation of the origin and a suitable rotation of axes, we can assume that the center
of mass is the origin and
\begin{alignat*}{1}
I(\Omega)=\int_{\Omega} |x|^2 dx.
\end{alignat*}
Obviously, we have
\begin{alignat*}{1}
I(\Omega)\geq \frac{n}{n+2}V(\Omega)\left(\frac{V(\Omega)}{\omega_n} \right)^{\frac{2}{n}}.
\end{alignat*}

We now fix a $k \geq 1$ and let $u_1 ,\ldots,u_k$ denote an orthonormal set of eigenfunctions of (\ref{DLE}) corresponding to the set of eigenvalues $\lambda_1(\Omega),\ldots,\lambda_k(\Omega)$. We consider the Fourier transform of each eigenfunction
\begin{alignat*}{1}
f_j(\xi)=\hat{u}_j(\xi)=(2\pi)^{-n/2}\int_{\Omega} u_j(x)e^{ix\xi}dx.
\end{alignat*}
From Plancherel's theorem  we have that  $f_1 ,.\ldots, f_k$ is an orthonormal set in $L^2(\mathbb{R}^n)$. Since  eigenfunctions $u_1 ,\ldots,u_k$ are also orthonormal in $L^2(\Omega)$, Bessel's inequality implies that for every $\xi\in \mathbb{R}^n$
\begin{alignat}{1}\label{S1}
\sum_{j=1}^k |f_j(\xi)|^2\leq (2\pi)^{-n}\int_{\Omega} |e^{ix\xi}|^2dx=(2\pi)^{-n}V(\Omega).
\end{alignat}
Since
\begin{alignat*}{1}
\nabla f_j(\xi)=(2\pi)^{-n/2}\int_{\Omega}ixu_j(x)e^{ix\xi}dx,
\end{alignat*}
we have
\begin{alignat*}{1}
\sum_{j=1}^k |\nabla f_j(\xi)|^2\leq (2\pi)^{-n/2}\int_{\Omega}|ixe^{ix\xi}|^2dx=(2\pi)^{-n}I(\Omega).
\end{alignat*}
By the boundary condition, we get
\begin{alignat*}{1}
\int_{\mathbb{R}^n}|\xi|^2|f_j(\xi)|^2d\xi=\int_{\Omega}|\nabla u_j (x)|^2dx=\lambda_j(\Omega)
\end{alignat*}
for each $1\leq j\leq k$. Set
\begin{alignat*}{1}
F(\xi)=\sum_{j=1}^k|f_j(\xi)|^2.
\end{alignat*}
From (\ref{S1}), we have
\begin{alignat}{1}
&0\leq F(\xi)\leq (2\pi)^{-n}V(\Omega),\\
|\nabla F(\xi)|\leq 2\left( \sum_{j=1}^k|f_j(\xi)|^2 \right)^{1/2} &\left( \sum_{j=1}^k|\nabla f_j(\xi)|^2 \right)^{1/2}\leq 2(2\pi)^{-n}\sqrt{V(\Omega)I(\Omega)}\label{T11111}
\end{alignat}
for each $\xi\in \mathbb{R}^n$.  We also get
\begin{alignat}{1}
\int_{\mathbb{R}^n}F(\xi)d\xi&=k,\\
\int_{\mathbb{R}^n} |\xi|^2F(\xi)d\xi&=\sum_{j=1}^k \lambda_j(\Omega).\label{T2222}
\end{alignat}

Assume (by approximating $F$) that the decreasing function $\phi: [0,+\infty)\rightarrow [0,(2\pi)^{-n}V(\Omega)]$  is absolutely continuous. Let $F^*(\xi) = \phi (|\xi|)$ denote the decreasing radial rearrangement of $F$.  Put $\mu(t)=|\{F^*>t \}|=|\{F>t\}|$.  It follows from the coarea formula that
\begin{alignat*}{1}
\mu(t)=\int_{t}^{(2\pi)^{-n}V(\Omega)} \int_{\{F=s \}}\frac{1}{|\nabla F|}d\sigma_sds.
\end{alignat*}
Since $F^*$ is radial, we have $\mu(\phi(s))=|\{F^*>\phi(s) \}|=\omega_n s^n$. Differentiating  both sides of the above equality, we have $n\omega_n s^{n-1}={\mu}'(\phi(s))\phi'(s)$ for almost all $s$. This together with (\ref{T11111}), $\rho=2(2\pi)^{-n}\sqrt{V(\Omega)I(\Omega)}$ and the isoperimetric inequality imply
\begin{alignat*}{1}
-\mu'(\phi(s)) &=\int_{\{F=\phi(s) \}} |\nabla F|^{-1} d\sigma_{\phi(s)}\\
& \geq \rho^{-1}\mathrm{Vol}_{n-1}(\{F=\phi(s)\})\\
& \geq \rho^{-1}n\omega_n s^{n-1}.
\end{alignat*}
For almost all $s$, we have
\begin{alignat}{1}\label{GI}
-\rho\leq \phi'(s)\leq 0.
\end{alignat}

Since the map $\xi \mapsto |\xi|^2$ is radial and increasing, applying (\ref{T2222}), we get
\begin{alignat}{1}\label{KI}
k=\int_{\mathbb{R}^n}F(\xi)d\xi=\int_{\mathbb{R}^n}F^*(\xi)d\xi=n\omega_n\int_0^{\infty}s^{n-1}\phi(s)ds
\end{alignat}
and
\begin{alignat}{1}\label{LE}
\sum_{j=1}^k\lambda_j(\Omega)=\int_{\mathbb{R}^n}|\xi|^2F(\xi)d\xi\geq \int_{\mathbb{R}^n}|\xi|^2F^*(\xi)d\xi=n\omega_n\int_0^\infty s^{n+1}\phi(s)ds.
\end{alignat}

The following lemma, which is one of  our key observations,  will be crucial in the proof of Theorem \ref{C1}.
\begin{Lemma}\label{ke1}
For any integer $n>0$  and positive real numbers $s$ and $\tau$ we have the following
equation:
\begin{alignat}{1}\label{eq:main}
ns^{n+2} - (n+2)\tau^2 s^n + 2\tau^{n+2}=\sum_{k=1}^{n} 2k s^{k-1} \tau^{n-k+1} (\tau-s)^2 + n(\tau-s)^2 s^n.
\end{alignat}
\end{Lemma}
\begin{Proofp}
We set $E$ as
\begin{alignat*}{1}
E =& \big[ ns^{n+2} - (n+2)\tau^2 s^n + 2\tau^{n+2} \big] \\\
&- (\tau-s)^2 \sum_{k=1}^{n} 2k s^{k-1} \tau^{n-k+1} - n(\tau-s)^2 s^n.
\end{alignat*}
If we can show that
\begin{equation}\label{eq:key}
ns^{n+2} - (n+2)\tau^2 s^n + 2\tau^{n+2} = (\tau-s)^2 \left( \sum_{k=1}^{n} 2k s^{k-1} \tau^{n-k+1} + n s^n \right),
\end{equation}
then substituting it into $E$ yields $E = 0$. Thus, it suffices to prove (\ref{eq:key}). Both sides of (\ref{eq:key}) are homogeneous polynomials of degree $n+2$ in $s$ and $\tau$. We compare coefficients of $\tau^j s^{n+2-j}$ for $j=0,1,\dots,n+2$.

The left-hand side is $L = ns^{n+2} - (n+2)\tau^2 s^n + 2\tau^{n+2}.$ Immediately,
\begin{itemize}
\item For $j=0$ (term $s^{n+2}$): coefficient $n$.
\item For $j=2$ (term $\tau^2 s^n$): coefficient $-(n+2)$.
\item For $j=n+2$ (term $\tau^{n+2}$): coefficient $2$.
\item For all other $j$: coefficient $0$.
\end{itemize}
Expand the right-hand side:
\begin{align*}
R &= (\tau-s)^2 \sum_{k=1}^{n} 2k s^{k-1} \tau^{n-k+1} + n(\tau-s)^2 s^n \\
&= (\tau^2 - 2\tau s + s^2) \sum_{k=1}^{n} 2k s^{k-1} \tau^{n-k+1} + n(\tau^2 - 2\tau s + s^2) s^n \\
&= \underbrace{\sum_{k=1}^{n} 2k s^{k-1} \tau^{n-k+3}}_{A} \;\; \underbrace{-\sum_{k=1}^{n} 4k s^{k} \tau^{n-k+2}}_{B} \;+\; \underbrace{\sum_{k=1}^{n} 2k s^{k+1} \tau^{n-k+1}}_{C} \\
&\quad + \underbrace{ n\tau^2 s^n - 2n\tau s^{n+1} + n s^{n+2}}_{D}.
\end{align*}
Now we fix $j$ and compute the coefficient $C_R(j):=A_j+B_j+C_j+D_j$ of $\tau^j s^{n+2-j}$ in the expansion of $R$. From the definition of $A$, $\tau$ has exponent $n-k+3$, $s$ has exponent $k-1$, and the total degree of $s$ and $\tau$ is $n+2$. For a given $j$, we have $n-k+3 = j$, i.e., $k = n+3-j$. Requirement $1 \le k \le n$ implies $3 \le j \le n+2$. Hence the coefficient of  $\tau^j s^{n+2-j}$ in $A$ is $2k = 2(n+3-j)$. Therefore
\[
A_j = \begin{cases}
2(n+3-j), & 3 \le j \le n+2, \\
0, & \text{otherwise}.
\end{cases}
\]

Similarly, we get
\[
B_j = \begin{cases}
-4(n+2-j), & 2 \le j \le n+1, \\
0, & \text{otherwise},
\end{cases}
\]
\[
C_j = \begin{cases}
2(n+1-j), & 1 \le j \le n, \\
0, & \text{otherwise}.
\end{cases}
\]
and
\[
D_j = \begin{cases}
n, & j=0, \\
-2n, & j=1, \\
n, & j=2, \\
0, & \text{otherwise}.
\end{cases}
\]
According to these equations, we have
Thus,
\[
C_R(j) =
\begin{cases}
n, & j=0, \\
0, & j=1, \\
-n-2, & j=2, \\
0, & 3 \le j \le n+1, \\
2, & j=n+2.
\end{cases}
\]
These exactly match the coefficients of the left-hand side $L$. Therefore, (\ref{eq:key}) holds.

\cvd
\end{Proofp}

The following lemma will be used in the proof of Theorem \ref{C1}.

\begin{Lemma}\label{KL}
Let $n\geq 2$, $\rho>0$ and $a$ be defined by (\ref{DA}). If $\psi: [0,+\infty)\rightarrow [0,+\infty)$ is a decreasing function (and absolutely continuous) satisfying
\begin{alignat*}{1}
-\rho\leq -\psi'(s)\leq 0
\end{alignat*}
and
\begin{alignat*}{1}
\int_0^\infty s^{n-1}\psi(s)ds=A.
\end{alignat*}
Then
\begin{alignat*}{1}
\int_0^\infty  s^{n+1}\psi(s)ds\geq & \frac{\psi(0)^{-\frac{2}{n}}}{n}(nA)^{\frac{n+2}{n}}+\frac{\psi(0)^{n+3}}{(n+2)(n+3)\rho^{n+2}}S_{n+3}\\
&-\frac{\psi(0)^{\frac{(n+2)(n-1)}{n}}}{n(n+1)\rho^n}(nA)^{\frac{2}{n}}S_{n+1},
\end{alignat*}
where
\begin{alignat*}{1}
S_j=(a+1)^j-a^j\geq 1.
\end{alignat*}

\end{Lemma}

\begin{Proofp}
We choose the function $\alpha \psi(\beta t)$ for appropriate $\alpha, \beta >0$, such that $\rho = 1$ and $\psi(0) = 1$. By \cite{M} we can also assume that $$B=\int_0^\infty s^{n+1}\psi(s)ds <\infty.$$ Set $q(s)=-\psi^{'}(s)$ for $s\geq 0$, we have $0\leq q(s)\leq 1$ and $\int_0^\infty q(s)=\psi(0)=1.$ Moreover, integration by parts implies that
\begin{alignat*}{1}
\int_0^\infty s^{n}q(s)ds=n\int_0^\infty s^{n-1}\psi(s)ds=nA
\end{alignat*}
and
\begin{alignat*}{1}
\int_0^\infty s^{n+2}q(s)ds\leq (n+2)B.
\end{alignat*}
Next, let $0\leq a < +\infty$ satisfies that
\begin{alignat}{1}\label{DA}
\int_a^{a+1} s^{n}ds=\int_0^\infty s^{n}q(s)ds=nA.
\end{alignat}
By the same argument as in  proofs of Lemma 1 in \cite{LY}, such a real number $a$ exists. From \cite{M}, we have
\begin{alignat}{1}\label{BI}
(n+2)B\geq \int_0^{\infty}s^{n+2}q(s)ds\geq\int_a^{a+1}s^{n+2}ds.
\end{alignat}
To estimate the last integral we take $\tau > 0$ to be chosen later. Applying  (\ref{BI}) and integrating  both sides of the following equation from $a$ to $a+1$
\begin{alignat*}{1}
\begin{split}
n&s^{n+2}-(n+2)\tau^2s^n+2\tau^{n+2}\\
=&\sum_{k=1}^{n} 2k s^{k-1} \tau^{n-k+1} (\tau-s)^2 +n(\tau-s)^2 s^n
\end{split}
\end{alignat*}
we get
\begin{alignat}{1}\label{ca1}
\begin{split}
n&(n+2)B-(n+2)\tau^2nA+2\tau^{n+2}\\
\geq& \int^{a+1}_a\sum_{k=1}^{n} 2k \tau^{n-k+1}s^{k-1} (\tau-s)^2ds +n\int^{a+1}_a(\tau-s)^2 s^nds.
\end{split}
\end{alignat}
For convenience, we define $S$ by:
\begin{align*}
S &= \int_{a}^{a+1} \Bigg[ \sum_{k=1}^{n} 2k \tau^{\,n-k+1} s^{k-1} (\tau-s)^2 \;+\; n s^{\,n} (\tau-s)^2 \Bigg] \, ds \\
  &= \int_{a}^{a+1} (\tau-s)^2 \; \Bigg( \sum_{k=1}^{n} 2k \tau^{\,n-k+1} s^{k-1} \;+\; n s^{\,n} \Bigg) \, ds .
\end{align*}
Let $x = s/\tau$.  Then
\[
\sum_{k=1}^{n} 2k \tau^{\,n-k+1} s^{k-1}
   = 2\tau^{\,n} \sum_{k=1}^{n} k x^{k-1}.
\]
Using the derivative of a geometric series,
\[
\sum_{k=1}^{n} k x^{k-1} = \frac{1-(n+1)x^{\,n}+n x^{\,n+1}}{(1-x)^2} \qquad (x\neq 1),
\]
we obtain
\begin{align*}
\sum_{k=1}^{n} 2k \tau^{\,n-k+1} s^{k-1}
   &= 2\tau^{\,n}\,
      \frac{\tau^2}{(\tau-s)^2}\,
      \Bigl(1-(n+1)\frac{s^{\,n}}{\tau^{\,n}}
            +n\frac{s^{\,n+1}}{\tau^{\,n+1}}\Bigr) \\[2mm]
   &= 2\,
      \frac{\tau^{\,n+2}-(n+1)\tau^2 s^{\,n}+n\tau s^{\,n+1}}{(\tau-s)^2}.
\end{align*}
Therefore, we get
\begin{align*}
S &= \int_{a}^{a+1} (\tau-s)^2\,
      \Bigg[ 2\,
             \frac{\tau^{\,n+2}-(n+1)\tau^2 s^{\,n}+n\tau s^{\,n+1}}
                  {(\tau-s)^2}
             \;+\; n s^{\,n} \Bigg] \, ds \\
  &= \int_{a}^{a+1} \Bigl[ 2\bigl(\tau^{\,n+2}-(n+1)\tau^2 s^{\,n}+n\tau s^{\,n+1}\bigr)
                          + n s^{\,n}(\tau-s)^2 \Bigr] \, ds .
\end{align*}
Now expanding $$n s^{\,n}(\tau-s)^2 = n\tau^2 s^{\,n} - 2n\tau s^{\,n+1} + n s^{\,n+2}$$ we get
\begin{align*}
S &= \int_{a}^{a+1} \Bigl[ 2\tau^{\,n+2}
                          -2(n+1)\tau^2 s^{\,n}
                          +2n\tau s^{\,n+1} \\
  &\qquad\qquad\;+\; n\tau^2 s^{\,n}
                          -2n\tau s^{\,n+1}
                          + n s^{\,n+2} \Bigr] \, ds \\[1mm]
  &= \int_{a}^{a+1} \Bigl[ 2\tau^{\,n+2}
                          -(n+2)\tau^2 s^{\,n}
                          + n s^{\,n+2} \Bigr] \, ds .
\end{align*}
Finally, since
\begin{align*}
\int_{a}^{a+1} 2\tau^{\,n+2} \, ds &= 2\tau^{\,n+2}, \\[2mm]
\int_{a}^{a+1} n s^{\,n+2} \, ds
    &= \frac{n}{n+3}\bigl[(a+1)^{\,n+3} - a^{\,n+3}\bigr], \\[2mm]
\int_{a}^{a+1} (n+2)\tau^2 s^{\,n} \, ds
    &= \frac{n+2}{n+1}\,\tau^2\,
       \bigl[(a+1)^{\,n+1} - a^{\,n+1}\bigr],
\end{align*}

\noindent we have
\[
S = 2\tau^{\,n+2}
    \;+\; \frac{n}{n+3}S_{n+3}
    \;-\; \frac{n+2}{n+1}\,\tau^2\,
          S_{n+1}.
\]
Putting this into (\ref{ca1}), we get
\begin{alignat*}{1}
\begin{split}
n(n+2)B-(n+2)\tau^2nA \geq \frac{n}{n+3}S_{n+3}-\frac{n+2}{n+1}\,\tau^2\,S_{n+1}.
\end{split}
\end{alignat*}
Let $\tau = (nA)^{1/n}$, we get
\begin{alignat*}{1}
\begin{split}
 B\geq \frac{1}{n}(nA)^{\frac{n+2}{n}} +\frac{1}{(n+2)(n+3)}S_{n+3}-\frac{1}{n(n+1)}(nA)^{\frac{2}{n}}S_{n+1},
\end{split}
\end{alignat*}
which implies that
\begin{alignat*}{1}
\begin{split}
\int_0^\infty  s^{n+1}\psi(s)ds\geq & \frac{\psi(0)^{-\frac{2}{n}}}{n}(nA)^{\frac{n+2}{n}}+\frac{\psi(0)^{n+3}}{(n+2)(n+3)\rho^{n+2}}S_{n+3}\\
&-\frac{\psi(0)^{\frac{(n+2)(n-1)}{n}}}{n(n+1)\rho^n}(nA)^{\frac{2}{n}}S_{n+1}.
\end{split}
\end{alignat*}
\cvd
\end{Proofp}

\subsection{Proof of Theorem \ref{C1}}
Applying Lemma \ref{KL} to the function $\phi$ with
$$A=(n\omega_n)^{-1}k,\quad \rho=2(2\pi)^{-n}\sqrt{V(\Omega)I(\Omega)}$$
and submitting it to  (\ref{LE}), we obtain
\begin{alignat}{1}\label{rl1}
\begin{split}
\sum_{i=1}^k \lambda_i
\geq &\omega_n {(nA)^{\frac{n+2}{n}}\phi(0)^{-\frac{2}{n}}}-\frac{\omega_n(nA)^{\frac{2}{n}}S_{n+1}}{(n+1)\rho^n}\phi(0)^{\frac{(n+2)(n-1)}{n}}\\
&+c_1\frac{n\omega_nS_{n+3}}{(n+2)(n+3)\rho^{n+2}}\phi(0)^{n+3},
\end{split}
\end{alignat}
where $c_1\in(0,1]$ is a positive constant and  $a$ is defined by
\begin{alignat}{1}\label{DA1}
\int_a^{a+1} \xi^{n}d\xi=\int_0^\infty -\xi^{n}\phi'(\xi)d\xi.
\end{alignat}

Notice that  $\rho=2(2\pi)^{-n}\sqrt{V(\Omega)I(\Omega)}$ and
$$I(\Omega)\geq \frac{n}{n+2}V(\Omega)\left(\frac{V(\Omega)}{\omega_n} \right)^{\frac{2}{n}}.$$
Consequently, we get
\begin{alignat}{1}\label{RHO}
\begin{split}
\rho\geq  \sqrt{2} (2\pi)^{-n}\omega_n^{-\frac{1}{n}}V(\Omega)^{\frac{n+1}{n}},
\end{split}
\end{alignat}
where we use the fact that
$$\left(\frac{n}{n+2}\right)^{\frac{1}{2}}\geq \frac{\sqrt{2}}{2},\,\,\,\quad\quad\forall\,\,n\geq 2.$$

Since $\phi(0)\in(0,(2\pi)^{-n}V(\Omega)],$ we define
\begin{alignat*}{1}
\begin{split}
g(t)=& g_1(t)-\frac{ (nA)^{\frac{2}{n}}S_{n+1}}{(n+1)\rho^n}t^{\frac{(n+2)(n-1)}{n}},\,\,\,t\in(0,(2\pi)^{-n}V(\Omega)]
\end{split}
\end{alignat*}
where
\begin{alignat*}{1}
\begin{split}
g_1(t)=& {(nA)^{\frac{n+2}{n}}t^{-\frac{2}{n}}}+c_1\frac{n S_{n+3}}{(n+2)(n+3)\rho^{n+2}}t^{n+3}.
\end{split}
\end{alignat*}
Obviously that $g(t)-g_1(t)$ is a decreasing function. Therefore, if we choose an appropriate $c_1$ such that $g_1(t)$ is a decreasing function on $t\in(0,(2\pi)^{-n}V]$, then $g(t)$ is also monotonically decreasing.

By direct calculation, we obtain the derivative of $g_1(t)$ as
\begin{alignat*}{1}
g'_1(t)=& -\frac{2}{n}{\left(\frac{k}{\omega_n}\right)^{\frac{n+2}{n}}t^{-\frac{2+n}{n}}}+c_1 \frac{n S_{n+3}}{(n+2)\rho^{n+2}}t^{n+2}.
\end{alignat*}
Hence, from
\begin{alignat*}{1}
c_1\leq  c_2:=\frac{2(n+2)}{n^2 S_{n+3}}\left(\frac{k}{\omega_n}\right)^{\frac{n+2}{n}} t^{-\frac{(n+2)(n+1)}{n}}\rho^{n+2},
\end{alignat*}
we infer that $g'_1(t)\leq 0$. Moreover, according to the fact that
\begin{alignat*}{1}
\frac{\omega_n^{\frac{4}{n}}}{(2\pi)^2}\leq \frac{1}{2},
\end{alignat*}
we derive that
\begin{alignat*}{1}
 c_2\geq \frac{2(n+2)}{n^2 S_{n+3}} (2\sqrt{2}\pi)^{n+2}\omega^{-{\frac{2(n+2)}{n}}}_n k^{\frac{n+2}{n}}\geq \frac{2^{n+3}(n+2)}{n^2 S_{n+3}}  k^{\frac{n+2}{n}}.
\end{alignat*}
Let
\begin{alignat*}{1}
 c_1=\min\left\{1, \frac{2^{n+3}(n+2)}{n^2 S_{n+3}}  k^{\frac{n+2}{n}}\right\},
\end{alignat*}
then we conclude that $g(t)$ is decreasing and $g(t)\geq g(\alpha)$, where $\alpha= {V(\Omega)}/{(2\pi)^n}$. Thus, we get
\begin{alignat*}{1}
\begin{split}
\sum_{i=1}^k \lambda_i
\geq & \frac{4\pi^2}{(\omega_nV(\Omega))^\frac{2}{n}}k^{\frac{n+2}{n}}-\frac{S_{n+1}}{(n+1)\rho^n}\omega^{\frac{n-2}{n}}_n\alpha^{\frac{(n+2)(n-1)}{n}}k^{\frac{2}{n}}\\
&+c_1\frac{nS_{n+3}}{(n+2)(n+3)\rho^{n+2}}\omega_n\alpha^{n+3},
\end{split}
\end{alignat*}
where $\alpha= {V(\Omega)}/{(2\pi)^n}, \rho=2(2\pi)^{-n}\sqrt{V(\Omega)I(\Omega)}$. Therefore, we complete the proof of Theorem \ref{C1}.

Moreover, we have the following  inequality.

\begin{Corollary}\label{KC111}
For any bounded domain $\Omega\subseteq \mathbb{R}^n$, $n\geq 2$ and any $ k \geq 1$ we have
\begin{alignat}{1}\label{MLE2}
\frac{1}{k}\sum_{i=1}^k \lambda_i \geq   \frac{n}{n+2}\frac{4\pi^2}{(\omega_nV(\Omega))^\frac{2}{n}}k^{\frac{2}{n}}+\bar{c}_n\frac{V(\Omega)}{I(\Omega)},\ \ \mathrm{for}\ k=1,2,\ldots,
\end{alignat}
where
\begin{alignat*}{1}
  \bar{c}_n=\frac{1}{n+2}\min\Big\{2k^{\frac{2}{n}},\frac{1}{12}\left( \frac{14}{5}+n  \left( \frac{1}{5}\right)^{n} \right)\Big\}.
\end{alignat*}
\end{Corollary}

\begin{proof}
Let $\tau=(nA)^{\frac{1}{n}}$.  If $n\geq2$, from Lemma \ref{ke1} and  (\ref{KL}), we get
\begin{alignat*}{1}
n&(n+2)B-(n+2)\tau^2nA+2\tau^{n+2} \\
\geq& 2\tau^n \int_a^{a+1}(s-\tau)^2 +4\tau^{n-1}\int_a^{a+1}s(s-\tau)^2ds+n\int_a^{a+1}(s-\tau)^2s^n ds.
\end{alignat*}
Define
\[
I_0(\tau) = \int_a^{a+1} (s-\tau)^2 \, ds, \quad
I_1(\tau) = \int_a^{a+1} s(s-\tau)^2 \, ds, \quad
I_n(\tau) = \int_a^{a+1} (s-\tau)^2 s^n \, ds
\]
and
\[
J(\tau) = 2\tau^n I_0(\tau) + 4\tau^{n-1} I_1(\tau) + n I_n(\tau).
\]
These integrals can be computed explicitly:
\begin{align*}
I_0(\tau) &= \frac{(a+1-\tau)^3 - (a-\tau)^3}{3}\geq\frac{1}{12}, \\
I_1(\tau) &= \frac{1}{12} \bigl( (a+1-\tau)^3(3a+3+4\tau) - (a-\tau)^3(3a+4\tau) \bigr).
\end{align*}
The ratio $\bar{s}_\tau = I_1(\tau)/I_0(\tau)$ is the weighted average of $s$ with weight $(s-\tau)^2$ on $[a,a+1]$. Since $f(s)=s^n$ is convex on $[a,a+1]$ for $n\ge2$ (as $f''(s)=n(n-1)s^{n-2}\ge0$), by Jensen's inequality with the positive weight function $(s-\tau)^2$, we have
\[
\frac{I_n(\tau)}{I_0(\tau)} \ge \left( \frac{I_1(\tau)}{I_0(\tau)} \right)^n = \bar{s}_\tau^{\, n}.
\]
Multiplying by $I_0(\tau)>0$ yields
\[
I_n(\tau) \ge \frac{\bigl( I_1(\tau) \bigr)^n}{\bigl( I_0(\tau) \bigr)^{n-1}}.
\]
Set $u = a - \tau$. Then
\begin{align*}
I_0(\tau) &= u^2 + u + \frac13, \\
I_1(\tau) &= \frac{1}{12}\Bigl[(u+1)^3(3a+3+4\tau) - u^3(3a+4\tau)\Bigr] \\
          &= \bigl(u^2+u+\tfrac13\bigr)\bigl(a+\tfrac12+\tfrac34\tau\bigr) + \frac{2u+1}{12}.
\end{align*}
Consequently,
\[
\bar{s}_\tau = a + \frac12 + \frac34\tau + r(u),\qquad
r(u) = \frac{2u+1}{12\bigl(u^2+u+\tfrac13\bigr)}.
\]
Due to \cite{JX2} we have $0.5<\tau<a+1, 0<\tau-a$, thus $u\in(-1,0)$. A simple analysis shows that $r(u)>-\frac{\sqrt{3}}{6}$ for $u>-1$. Hence
\[
\bar{s}_\tau > a + \frac12   -\frac{\sqrt{3}}{6}\geq \frac{4a}{5}+\frac{1+a}{5}\geq\frac{1}{5}\tau. \tag{A}
\]
Because $x\mapsto x^n$ is increasing for $x>0$, we get
\[
I_n(\tau) > I_0(\tau)\left( \frac{1}{5}\tau\right)^n.
\]
Define $h(s)=2\tau^n+4\tau^{n-1}s+ns^n$. Since $n\ge2$, $h$ is convex on $[a,a+1]$ because
$h''(s)=n^2(n-1)s^{n-2}\ge0$. Applying Jensen's inequality with weight $(s-\tau)^2$ gives
\[
\frac{J(\tau)}{I_0(\tau)} \ge h(\bar{s}_\tau) = 2\tau^n + 4\tau^{n-1}\bar{s}_\tau + n\bar{s}_\tau^{\, n}.
\]
Using the lower bound (A) for $\bar{s}_\tau$ and the fact that $h$ is increasing (its derivative $h'(s)=4\tau^{n-1}+n^2s^{n-1}>0$), we obtain
\[
h(\bar{s}_\tau) > 2\tau^n + 4\tau^{n-1}\Bigl(\frac{1}{5}\tau\Bigr)
               + n \Bigl(\frac{1}{5}\tau\Bigr)^{\!n}.
\]
Multiplying by $I_0(\tau)>0$ yields
\[
J(\tau) > I_0(\tau) \left( 2\tau^n + 4\tau^{n-1} \Bigl(\frac{1}{50}\tau \Bigr)
               + n \Bigl( \frac{1}{5}\tau \Bigr)^{\!n}  \right).
\]
Hence,
\[
J(\tau) > \frac{1}{12}\left( \frac{14}{5}+n  \left( \frac{1}{5}\right)^{n} \right)  \tau^n.
\]
Therefore, we arrive at
\begin{alignat*}{1}
n(n+2)B-(n+2)\tau^2nA+2\tau^{n+2} \geq& \frac{1}{12}\left( \frac{14}{5}+n  \left( \frac{1}{5}\right)^{n} \right)\tau^n.
\end{alignat*}
According to the similar discussion as in \cite{M}, we obtain that
\begin{alignat*}{1}
\frac{1}{k}\sum_{i=1}^k \lambda_i \geq   \frac{n}{n+2}\frac{4\pi^2}{(\omega_nV(\Omega))^\frac{2}{n}}k^{\frac{2}{n}}+\bar{c}_n\frac{V(\Omega)}{I(\Omega)},\ \ \mathrm{for}\ k=1,2,\ldots,
\end{alignat*}
where
\begin{alignat*}{1}
 \bar{c}_n=\frac{1}{n+2}\min\Big\{2k^{\frac{2}{n}},\frac{1}{12}\left( \frac{14}{5}+n  \left( \frac{1}{5}\right)^{n} \right)\Big\}.
\end{alignat*}

\cvd
\end{proof}

\begin{Remark}
Obviously, we have
\begin{alignat*}{1}
 \bar{c}_n=\frac{1}{n+2}\min\Big\{2k^{\frac{2}{n}},\frac{1}{12}\left( \frac{14}{5}+n  \left( \frac{1}{5}\right)^{n} \right)\Big\}>\frac{1}{24(n+2)}.
\end{alignat*}
Therefore, (\ref{MLE2}) is sharper than Melas' inequality (\ref{MLE}).
\end{Remark}

\section{Lower bounds for Dirichlet eigenvalues of the  poly-Laplacian}
In this section, we will prove  the following lower bound for Dirichlet eigenvalues of the poly-Laplacian.

\begin{Theorem}\label{C12}
For any bounded domain $\Omega\subseteq \mathbb{R}^n$, and $l,n\geq 2$,  we have
\begin{alignat}{1}\label{MT2lop}
\begin{split}
\frac{1}{k}\sum_{i=1}^k \Lambda_i
\geq &\omega_n\alpha^{-\frac{2l}{n}} \left(\frac{k}{\omega_n} \right)^{\frac{2l}{n}}
+\bar{c}_1\frac{n\omega_n\alpha^{2l+n+1}}{(2l+n)(n+2l+1)\rho^{2l+n}} S_{n+2l+1}k^{-1}\\
&-\frac{\omega_n\alpha^{\frac{n^2+n-2l}{n}}}{(n+1)\rho^{n}} \left(\frac{k}{\omega_n} \right)^{\frac{2l-n}{n}} S_{n+1},
\end{split}
\end{alignat}
where $S_{i}=(a+1)^i-a^i$, $a$ is defined by (\ref{DA12}), $\alpha,\rho$ are defined in (\ref{two quan.}),
\begin{alignat*}{1}
\bar{c}_1\leq& \{1,\bar{c}_2\},\,\,\,\,\,\mathrm{if}\,\, n^2+n>2l,\\
\bar{c}_1\leq& \{1,\bar{c}_3\},\,\,\,\,\,\mathrm{if}\,\, n^2+n\leq2l,
\end{alignat*}
and $\bar{c}_2,\bar{c}_3$ are defined by (\ref{dobc2}) and (\ref{dobc3}).
\end{Theorem}
\begin{Remark}
Theorem \ref{MT2lop} improves Theorem 4.1 in \cite{JX} and Theorem 1.2 in \cite{JX2}.
\end{Remark}

We fixed a $k\geq 1$ and let $\{\bar{u}_j\}_{i=1}^{k}$ be the orthonormal eigenfunctions corresponding to the eigenvalues $\{\Lambda_j\}_{j=1}^k$ for the following Dirichlet problem of the poly-Laplacian
\begin{equation}\label{EOE1}
\begin{cases}
(-\Delta)^l \bar{u}_j=\Lambda _j\bar{u}_j  \ & \mathrm{in}\ \Omega, \\
\bar{u}_j=\frac{\partial \bar{u}_j}{\partial \nu}=\cdots=\frac{\partial^{l-1} \bar{u}_j}{\partial \nu^{l-1}} = 0  \  & \mathrm{on}\ \partial\Omega.
\end{cases}
\end{equation}

According to the discussion in \cite{JX2}, we consider the Fourier transform of each eigenfunction
\begin{alignat}{1}
\bar{f}_j(\xi)=\hat{\bar{u}}_j(\xi)=(2\pi)^{-n/2}\int_{\Omega} \bar{u}_j(x)e^{ix\cdot\xi}dx.
\end{alignat}
From Plancherel's theorem, we know that  $\bar{f}_1 ,.\ldots, \bar{f}_k$ is an orthonormal set in $L^2(\mathbb{R}^n)$, which impies
\begin{alignat*}{1}
\int_{\Omega}\bar{f}_i(\xi)\bar{f}_j(\xi)dx=
\begin{cases}
0,  \ & \mathrm{if}\ i\neq j, \\
1,  \  & \mathrm{if}\ i=j.
\end{cases}
\end{alignat*}

By using integration by parts, we have
\begin{alignat*}{1}
\sum_{i=1}^k\Lambda _i=&\sum_{i=1}^k\int_{\Omega}u_i(x)(-\Delta)^{l} u_i(x)dx=\sum_{i=1}^k\sum_{j_1,\cdots,j_l=1}^n\int_{\mathbb{R}^n}\bigg( \frac{\partial^lu_{i}(x)}{\partial x_{j_1}\cdots\partial x_{j_l}} \bigg)^2dx.
\end{alignat*}
According to the Fourier transform and the Parseval's indentity, we can rewrite the above equation as
\begin{alignat}{1}\label{1le}
\begin{split}
\sum_{i=1}^k\Lambda _i=&\sum_{i=1}^k\sum_{j_1,\cdots,j_l=1}^n\int_{\mathbb{R}^n}\bigg|(2\pi)^{-\frac{n}{2}}\int_{\Omega}\xi_{j_1}\cdots\xi_{j_l}u_i(x)e^{ix\cdot\xi}dx \bigg|^2d\xi.
\end{split}
\end{alignat}

Set $Q(\xi)=\sum_{j=1}^k|\bar{f}_j(\xi)|^2$. From \cite{JX2}, we have $0\leq Q(\xi)\leq (2\pi)^{-n}V(\Omega)$
and
\begin{alignat}{1}
\begin{split}
|\nabla Q(\xi)|\leq& 2(2\pi)^{-n}\sqrt{V(\Omega)I(\Omega)}\label{T1}
\end{split}
\end{alignat}
for each $\xi\in \mathbb{R}^n$.  Combing the definition of $Q(\xi)$ with (\ref{1le}), we obtain the following equations
\begin{alignat}{1}
\int_{\mathbb{R}^n}Q(\xi)d\xi&=k,\\
\int_{\mathbb{R}^n} |\xi|^{2l}Q(\xi)d\xi&=\sum_{j=1}^k \Lambda_j(\Omega).\label{T2}
\end{alignat}

Next, we will introduce the decreasing radial rearrangement function. For convenience, we adhere the same notations in \cite{JX2}.  Let $Q^*(\xi) = \bar{\phi} (|\xi|)$ denote the decreasing radial rearrangement of $Q(\xi)$. By approximating $Q(\xi)$, we assume that  $ \bar{\phi}: [0,+\infty)\rightarrow [0,(2\pi)^{-n}V(\Omega)]$  is an absolutely continuous and  decreasing function. Using similar discussions in \cite{JX2}, we have
\begin{alignat}{1}
-\rho\leq  \bar{\phi}{'}(s)\leq 0
\end{alignat}
for almost every $s$.

Notice that (\ref{T2}) implies that
\begin{alignat}{1}
k=\int_{\mathbb{R}^n}Q(\xi)d\xi=\int_{\mathbb{R}^n}Q^*(\xi)d\xi=n\omega_n\int_0^{\infty}s^{n-1}\bar{\phi}(s)ds
\end{alignat}
and
\begin{alignat}{1}\label{LE22}
\sum_{i=1}^k\Lambda_j(\Omega)=\int_{\mathbb{R}^n}|\xi|^{2l}Q(\xi)d\xi\geq \int_{\mathbb{R}^n}|\xi|^{2l}Q^*(\xi)d\xi=n\omega_n\int_0^\infty s^{2l+n-1}\bar{\phi}(s)ds
\end{alignat}
since $\xi \rightarrow |\xi|^2$ is radial and increasing.

Firstly, we will prove the following key equation.

\begin{Lemma}\label{KL22}
For any positive integers $d\geq 2,q\geq 2$ and positive numbers $s,\tau$, we have
\begin{alignat}{1}\label{kefpl}
\begin{split}
d&s^{d+q} - (d+q)s^d\tau^q + q\tau^{d+q} = q\sum_{k=1}^{d} k s^{k-1} \tau^{d+q-1-k} (s-\tau)^2 \\
+& d (s-\tau)^2 \sum_{j=0}^{q-2} (q-1-j) s^{d+j} \tau^{q-2-j}.
\end{split}
\end{alignat}

\end{Lemma}

\begin{proof}
Let $x = s/\tau$, so that $s = x\tau$. Divide the left-hand side of (\ref{kefpl}) by $\tau^{d+q}$ and denote
\[
L = \frac{1}{\tau^{d+q}} \left[ ds^{d+q} - (d+q)s^d\tau^q + q\tau^{d+q} - q\sum_{k=1}^{d} k s^{k-1} \tau^{d+q-1-k} (s-\tau)^2 \right].
\]
Substituting $s = x\tau$, we obtain
\[
L = d x^{d+q} - (d+q) x^d + q - q (x-1)^2 \sum_{k=1}^{d} k x^{k-1}.
\]

Compute the sum $\sum_{k=1}^{d} k x^{k-1}$. For $x \neq 1$, we use the formula
\[
\sum_{k=1}^{d} k x^{k-1} = \frac{d x^{d+1} - (d+1) x^d + 1}{(1-x)^2}.
\]
Substituting this into $L$ yields
\begin{align*}
L &= d x^{d+q} - (d+q) x^d + q - q (x-1)^2 \cdot \frac{d x^{d+1} - (d+1) x^d + 1}{(1-x)^2} \\
&= d x^{d+q} - (d+q) x^d + q - q \left( d x^{d+1} - (d+1) x^d + 1 \right) \\
&= d x^{d+q} - q d x^{d+1} + \left[ -(d+q) + q(d+1) \right] x^d \\
&= d x^{d+q} - q d x^{d+1} + d(q-1) x^d \\
&= d x^d \left( x^q - q x + (q-1) \right).
\end{align*}

Now we prove the polynomial identity
\[
x^q - q x + (q-1) = (x-1)^2 \sum_{j=0}^{q-2} (q-1-j) x^j.
\]
Let $P(x) = x^q - q x + (q-1)$. Observe that $P(1) = 0$ and $P'(x) = q x^{q-1} - q$, so $P'(1) = 0$. Hence $x=1$ is a double root of $P(x)$, and $P(x)$ is divisible by $(x-1)^2$. Write
\[
P(x) = (x-1)^2 \sum_{j=0}^{q-2} a_j x^j.
\]
Next we determine the coefficients $a_j$. Expanding the right-hand side
\[
(x-1)^2 \sum_{j=0}^{q-2} a_j x^j = (x^2 - 2x + 1) \sum_{j=0}^{q-2} a_j x^j = \sum_{j=0}^{q-2} a_j x^{j+2} - 2 \sum_{j=0}^{q-2} a_j x^{j+1} + \sum_{j=0}^{q-2} a_j x^j.
\]
Re-index the sums:
\begin{align*}
\sum_{j=0}^{q-2} a_j x^{j+2} &= \sum_{k=2}^{q} a_{k-2} x^k, \\
\sum_{j=0}^{q-2} a_j x^{j+1} &= \sum_{k=1}^{q-1} a_{k-1} x^k, \\
\sum_{j=0}^{q-2} a_j x^j &= \sum_{k=0}^{q-2} a_k x^k.
\end{align*}
Thus,
\[
P(x) = \sum_{k=2}^{q} a_{k-2} x^k - 2 \sum_{k=1}^{q-1} a_{k-1} x^k + \sum_{k=0}^{q-2} a_k x^k.
\]
Collecting like powers of $x$:
\begin{itemize}
\item For $k=0$: coefficient is $a_0$.
\item For $k=1$: coefficient is $-2a_0 + a_1$.
\item For $2 \le k \le q-2$: coefficient is $a_{k-2} - 2a_{k-1} + a_k$.
\item For $k = q-1$: coefficient is $a_{q-3} - 2a_{q-2}$ (since the third sum ends at $k=q-2$).
\item For $k = q$: coefficient is $a_{q-2}$.
\end{itemize}
We require $P(x) = x^q - q x + (q-1)$. Comparing coefficients:
\begin{align*}
&a_0 = q-1, \\
&-2a_0 + a_1 = -q, \\
&a_{k-2} - 2a_{k-1} + a_k = 0 \quad \text{for } 2 \le k \le q-2, \\
&a_{q-3} - 2a_{q-2} = 0, \\
&a_{q-2} = 1.
\end{align*}
From $a_0 = q-1$ and $-2a_0 + a_1 = -q$, we get $a_1 = -q + 2(q-1) = q-2$.
From $a_{q-2} = 1$ and the recurrence $a_k = 2a_{k-1} - a_{k-2}$ (for $2 \le k \le q-2$), we see that the sequence $\{a_j\}$ is arithmetic with first term $a_0 = q-1$ and common difference $-1$. Hence $a_j = q-1 - j$ for $j=0,1,\dots,q-2$. It is easy to verify that this satisfies all conditions: $a_{q-2} = q-1 - (q-2) = 1$, and $a_{q-3} = 2$, so $a_{q-3} - 2a_{q-2} = 0$.

Therefore,
\[
x^q - q x + (q-1) = (x-1)^2 \sum_{j=0}^{q-2} (q-1-j) x^j.
\]

Substituting this into the expression for $L$, we obtain
\[
L = d x^d (x-1)^2 \sum_{j=0}^{q-2} (q-1-j) x^j.
\]
Reverting to $s$ and $\tau$ via $x = s/\tau$ and $(x-1)^2 = (s-\tau)^2/\tau^2$, we have
\begin{align*}
\text{LHS} &= \tau^{d+q} L = \tau^{d+q} \cdot d \left( \frac{s}{\tau} \right)^d \cdot \frac{(s-\tau)^2}{\tau^2} \sum_{j=0}^{q-2} (q-1-j) \left( \frac{s}{\tau} \right)^j \\
&= d s^d (s-\tau)^2 \tau^{d+q-d-2} \sum_{j=0}^{q-2} (q-1-j) s^j \tau^{-j} \\
&= d (s-\tau)^2 \sum_{j=0}^{q-2} (q-1-j) s^{d+j} \tau^{q-2-j}.
\end{align*}
This is exactly the right-hand side. The identity also holds trivially when $x=1$ (i.e., $s=\tau$), as both sides vanish. This completes the proof.

\cvd
\end{proof}

The following lemma will be used in the proof of Theorem \ref{C12}.

\begin{Lemma}\label{KL22}
Let $n\geq 2$, $l\geq 2$, $\rho>0$ and $a$ be defined by (\ref{DA}). If $\psi: [0,+\infty)\rightarrow [0,+\infty)$ is a decreasing function (and absolutely continuous) satisfying
\begin{alignat*}{1}
-\rho\leq -\psi'(s)\leq 0
\end{alignat*}
and
\begin{alignat*}{1}
\int_0^\infty s^{n-1}\psi(s)ds=A.
\end{alignat*}
Then
\begin{alignat*}{1}
\int_0^\infty  s^{2l+n-1}\psi(s)ds\geq & \frac{\psi(0)^{-\frac{2l}{n}}}{n}(nA)^{\frac{2l+n}{n}}
+\frac{\psi(0)^{2l+n+1}}{(2l+n)(n+2l+1)\rho^{2l+n}} S_{n+2l+1}\\
&-\frac{\psi(0)^{\frac{n^2+n-2l}{n}}}{n(n+1)\rho^{n}}(nA)^{\frac{2l}{n}} S_{n+1}.
\end{alignat*}
where
\begin{alignat*}{1}
S_j=(a+1)^j-a^j\geq 1.
\end{alignat*}

\end{Lemma}

\begin{proof}Using similar strategy with that used in the proof of Lemma \ref{KL},  we choose the function $\alpha \psi(\beta t)$ for appropriate $\alpha, \beta >0$, such that $\rho = 1$ and $\psi(0) = 1$. We also assume that $$B=\int_0^\infty s^{2l+n-1}\psi(s)ds <\infty.$$ If we let $q(s)=-\psi{'}(s)$ for $s\geq 0$, we have $0\leq q(s)\leq 1$ and $\int_0^\infty q(s)=\psi(0)=1.$ Moreover, integration by parts implies that
\begin{alignat*}{1}
\int_0^\infty s^{n}q(s)ds=n\int_0^\infty s^{n-1}\psi(s)ds=nA
\end{alignat*}
and
\begin{alignat*}{1}
\int_0^\infty s^{2l+n}q(s)ds\leq (2l+n)B.
\end{alignat*}
Next, let $0\leq a < +\infty$ satisfies that
\begin{alignat}{1}\label{DA}
\int_a^{a+1} s^{n}ds=\int_0^\infty s^{n}q(s)ds=nA.
\end{alignat}
According to Lemma \ref{KL22}, we get
\begin{alignat*}{1}
\begin{split}
n&s^{n+2l} - (n+2l)s^n\tau^{2l} + 2l\tau^{n+2l} - 2l\sum_{k=1}^{n} k s^{k-1} \tau^{n+2l-1-k} (s-\tau)^2 \\
=& n (s-\tau)^2 \sum_{j=0}^{2l-2} (2l-1-j) s^{n+j} \tau^{2l-2-j},\,\,\,\,\,\,\,s\in[a,a+1].
\end{split}
\end{alignat*}
Integrating the both sides of the above equation on $[a,a+1]$, we get
\begin{alignat}{1}\label{KEOG}
\begin{split}
n&(2l+n)B-(2l+n)(nA)\tau^{2l}+2l\tau^{2l+n}\\
\geq&2l\sum_{k=1}^{n} k \tau^{n+2l-1-k}\int^{a+1}_a s^{k-1} (s-\tau)^2ds \\
&+ n \sum_{j=0}^{2l-2} (2l-1-j)\tau^{2l-2-j}\int^{a+1}_a s^{n+j}(s-\tau)^2ds.
\end{split}
\end{alignat}
For convenience, we define
\begin{alignat*}{1}
S =& 2l\sum_{k=1}^{n} k \tau^{n+2l-1-k}\int^{a+1}_a s^{k-1} (s-\tau)^2\,ds\\
&+ n \sum_{j=0}^{2l-2} (2l-1-j)\tau^{2l-2-j}\int^{a+1}_a s^{n+j}(s-\tau)^2\,ds.
\end{alignat*}
Since the sums are finite, we may interchange the order of summation and integration:
\[
S = \int_a^{a+1} (s-\tau)^2\Bigg(2l\sum_{k=1}^{n} k\tau^{n+2l-1-k}s^{k-1}
+ n\sum_{j=0}^{2l-2}(2l-1-j)\tau^{2l-2-j}s^{n+j}\Bigg)ds.
\]
Let \(x = s/\tau\).  Then
\[
2l\sum_{k=1}^{n} k\tau^{n+2l-1-k}s^{k-1}
= 2l\tau^{n+2l-2}\sum_{k=1}^{n} k x^{k-1},
\]
and
\[
n\sum_{j=0}^{2l-2}(2l-1-j)\tau^{2l-2-j}s^{n+j}
= n\tau^{n+2l-2}x^n\sum_{j=0}^{2l-2}(2l-1-j)x^{j}.
\]

Using the formulas
\begin{align*}
\sum_{k=1}^{n} k x^{k-1}=&\frac{1-(n+1)x^n+nx^{n+1}}{(1-x)^2},\\
\sum_{j=0}^{2l-2}(2l-1-j)x^{j}=& \frac{(2l-1)-2l x+x^{2l}}{(1-x)^2},
\end{align*}
and noting that \((s-\tau)^2=\tau^2(1-x)^2\), we obtain
\begin{align*}
&(s-\tau)^2\Bigg(2l\sum_{k=1}^{n} k\tau^{n+2l-1-k}s^{k-1}
+n\sum_{j=0}^{2l-2}(2l-1-j)\tau^{2l-2-j}s^{n+j}\Bigg)\\
&=\tau^{n+2l}\Bigg[2l\bigl(1-(n+1)x^n+nx^{n+1}\bigr)
+nx^n\bigl((2l-1)-2lx+x^{2l}\bigr)\Bigg].
\end{align*}
Expanding the bracket gives
\begin{align*}
&2l-(2l(n+1)-n(2l-1))x^n+(2ln-2ln)x^{n+1}+nx^{n+2l}\\
&=2l-(n+2l)x^n+nx^{n+2l}.
\end{align*}
Since \(x^n=s^n/\tau^n\) and \(x^{n+2l}=s^{n+2l}/\tau^{n+2l}\), we finally have
\begin{align*}
&(s-\tau)^2\Bigg(2l\sum_{k=1}^{n} k\tau^{n+2l-1-k}s^{k-1}
+n\sum_{j=0}^{2l-2}(2l-1-j)\tau^{2l-2-j}s^{n+j}\Bigg)\\
=&2l\tau^{n+2l}+ns^{n+2l}-(n+2l)\tau^{2l}s^n.
\end{align*}
Thus
\[
S=\int_a^{a+1}\!\bigl(2l\tau^{n+2l}+ns^{n+2l}-(n+2l)\tau^{2l}s^n\bigr)ds.
\]
Integrating term by term,
\begin{align*}
\int_a^{a+1}2l\tau^{n+2l}\,ds &=2l\tau^{n+2l},\\
\int_a^{a+1}ns^{n+2l}\,ds &=\frac{n}{n+2l+1}\bigl((a+1)^{n+2l+1}-a^{n+2l+1}\bigr),\\
\int_a^{a+1}(n+2l)\tau^{2l}s^n\,ds
&=\frac{n+2l}{n+1}\tau^{2l}\bigl((a+1)^{n+1}-a^{n+1}\bigr).
\end{align*}
Therefore,
\begin{align*}
S=2l\tau^{n+2l}
+\frac{n}{n+2l+1}S_{n+2l+1}-\frac{n+2l}{n+1}\tau^{2l}S_{n+1},
\end{align*}
where $S_j=(a+1)^j-a^j.$ Putting this equation into (\ref{KEOG}), we get
\begin{alignat*}{1}
\begin{split}
n(2l+n)B\geq(2l+n)(nA)\tau^{2l}+\frac{n}{n+2l+1}S_{n+2l+1}-\frac{n+2l}{n+1}\tau^{2l}S_{n+1}.
\end{split}
\end{alignat*}
Let $\tau = (nA)^{1/n}$, we get
\begin{alignat*}{1}
\begin{split}
n(2l+n)B\geq(2l+n)(nA)^{1+\frac{2l}{n}}+\frac{n}{n+2l+1}S_{n+2l+1}-\frac{n+2l}{n+1}(nA)^{\frac{2l}{n}}S_{n+1},
\end{split}
\end{alignat*}
which equals to
\begin{alignat*}{1}
\begin{split}
 B\geq\frac{1}{n}(nA)^{\frac{2l+n}{n}}+\frac{1}{(2l+n)(n+2l+1)}S_{n+2l+1}-\frac{1}{n(n+1)}(nA)^{\frac{2l}{n}}S_{n+1}.
\end{split}
\end{alignat*}
Hence, we obtain that
\begin{alignat*}{1}
\int_0^\infty  s^{2l+n-1}\psi(s)ds\geq & \frac{\psi(0)^{-\frac{2l}{n}}}{n}(nA)^{\frac{2l+n}{n}}
+\frac{\psi(0)^{2l+n+1}}{(2l+n)(n+2l+1)\rho^{2l+n}} S_{n+2l+1}\\
&-\frac{\psi(0)^{\frac{n^2+n-2l}{n}}}{n(n+1)\rho^{n}}(nA)^{\frac{2l}{n}} S_{n+1}.
\end{alignat*}

\cvd
\end{proof}

\subsection{Proof of Theorem \ref{C12}}
Applying Lemma \ref{KL22} to the function $\bar{\phi}$ with
$$A=(n\omega_n)^{-1}k,\quad \rho=2(2\pi)^{-n}\sqrt{V(\Omega)I(\Omega)}$$
and submitting it to  (\ref{LE22}), we obtain
\begin{alignat}{1}\label{rl1}
\begin{split}
\sum_{i=1}^k \Lambda_i
\geq &\omega_n\bar{\phi}(0)^{-\frac{2l}{n}} (nA)^{\frac{2l+n}{n}}
+\bar{c}_1\frac{n\omega_n\bar{\phi}(0)^{2l+n+1}}{(2l+n)(n+2l+1)\rho^{2l+n}} S_{n+2l+1}\\
&-\frac{\omega_n\bar{\phi}(0)^{\frac{n^2+n-2l}{n}}}{(n+1)\rho^{n}}(nA)^{\frac{2l}{n}} S_{n+1},
\end{split}
\end{alignat}
where $\bar{c}_1\in(0,1]$ is a positive constant and  $a$ is defined by
\begin{alignat}{1}\label{DA12}
\int_a^{a+1} \xi^{n}d\xi=\int_0^\infty -\xi^{n}\bar{\phi}'(\xi)d\xi.
\end{alignat}
Since $\bar{\phi}(0)\in(0,(2\pi)^{-n}V(\Omega)],$  when $n^2+n\geq 2l$, we define
\begin{alignat*}{1}
\begin{split}
G(t)=& G_1(t)-\frac{t^{\frac{n^2+n-2l}{n}}}{(n+1)\rho^{n}}(nA)^{\frac{2l}{n}} S_{n+1}
\end{split}
\end{alignat*}
where
\begin{alignat*}{1}
\begin{split}
G_1(t)=t^{-\frac{2l}{n}} (nA)^{\frac{2l+n}{n}}+\bar{c}_1\frac{nt^{2l+n+1}}{(2l+n)(n+2l+1)\rho^{2l+n}} S_{n+2l+1}.
\end{split}
\end{alignat*}
Since  $n^2+n\geq 2l$, then $G(t)-G_1(t)$ is a decreasing function. Therefore, if we choose an appropriate $\bar{c}_1$ such that $G_1(t)$ is a decreasing function on $t\in(0,(2\pi)^{-n}V]$, then $G(t)$ is also monotonically decreasing. By direct calculation,  when
\begin{alignat}{1}\label{dobc2}
\begin{split}
 \bar{c}_1\leq \bar{c}_2:=\frac{2l\rho^{2l+n}}{n^2S_{n+2l+1}}\left( \frac{k}{\omega_n}\right)^{\frac{2l+n}{n}}\alpha^{-\frac{(n+1)(2l+n)}{n}},
\end{split}
\end{alignat}
$G'_1(t)<0$ which implies that  $G'(t)<0$ on $(0,(2\pi)^{-n}V(\Omega)]$.

If $n^2+n\leq 2l$,  then $G(t)-G_1(t)$ is an increasing function. Hence, we should discuss the tonicity of the whole function $G$. Since
\begin{alignat*}{1}
\begin{split}
G'(t)=&-\frac{2l}{n}t^{-1-\frac{2l}{n}} (nA)^{\frac{2l+n}{n}}+\bar{c}_1\frac{nt^{2l+n}}{(2l+n)\rho^{2l+n}} S_{n+2l+1}\\
&+\frac{2l-n^2-n}{n}\frac{t^{\frac{n^2-2l}{n}}}{(n+1)\rho^{n}}(nA)^{\frac{2l}{n}} S_{n+1},
\end{split}
\end{alignat*}
when $\bar{c}\leq \bar{c}_3$, then $G'(t)<0$ on $(0,(2\pi)^{-n}V(\Omega)]$, where
\begin{alignat}{1}\label{dobc3}
\begin{split}
\bar{c}_3:=\frac{(2l+n)\rho^{2l+n}}{nS_{2l+n+1}} \left(\frac{2lk}{n\omega_n}\alpha^{-\frac{(2l+n)(n+1)}{n}}
 +\frac{n^2+n-2l}{n(n+1)\rho^{n}}S_{n+1} \alpha^{\frac{n^2-2l}{n}} \right)\left( \frac{k}{\omega_n}\right)^{\frac{2l}{n}}.
\end{split}
\end{alignat}
This completes the proof of Theorem \ref{C12}.

Moreover, we also have the following  result.

\begin{Corollary}\label{CCLLL2}
For any bounded domain $\Omega\subseteq \mathbb{R}^n$, if $n\geq 2$, $l\geq 2$ and $ k \geq 1$ we have
\begin{alignat*}{1}
\begin{split}
\sum_{i=1}^k \Lambda_i
\geq &\frac{n}{n+2l}\omega_n\alpha^{-\frac{2l}{n}} \left( \frac{k}{\omega_n}\right)^{\frac{2l+n}{n}}
+\bar{C}_n\frac{\omega_n\alpha^{\frac{2n+2-2l}{n}}}{(2l+n)\rho^{2}}\left( \frac{k}{\omega_n}\right)^{\frac{2l+n-2}{n}}
\end{split}
\end{alignat*}
where
\begin{alignat*}{1}
\bar{C}_n= &Q\Bigl(\frac{1}{5}\Bigr),\,\,\,\,\,\,\mathrm{if}\,\,\,\,n+1-l\leq 0,\\
 \bar{C}_n=&\min\Big\{ \frac{2nl}{n+1-l}k^{\frac{2}{n}} ,Q\Bigl(\frac{1}{5}\Bigr)\Big\},\,\,\,\,\,\,\mathrm{if}\,\,\,\,n+1-l> 0
\end{alignat*}
and
\begin{alignat*}{1}
 Q(s)=&2l +4l s+n(2l-1) s^{\,n}+n(2l-2) s^{n+1}+n(2l-3)s^{n+2}.
\end{alignat*}
\end{Corollary}

\begin{proof}
According to the fact that $n\geq 2$, $l\geq 2$ and  (\ref{KEOG}, we get
\begin{alignat}{1}\label{KEOG2}
\begin{split}
n&(2l+n)B-(2l+n)(nA)\tau^{2l}+2l\tau^{2l+n}\\
\geq&2l\sum_{k=1}^{n} k \tau^{n+2l-1-k}\int^{a+1}_a s^{k-1} (s-\tau)^2ds \\
&+ n \sum_{j=0}^{2l-2} (2l-1-j)\tau^{2l-2-j}\int^{a+1}_a s^{n+j}(s-\tau)^2ds\\
\geq &2l\sum_{k=1}^{2} k \tau^{n+2l-1-k}\int^{a+1}_a s^{k-1} (s-\tau)^2ds \\
&+ n \sum_{j=0}^{2} (2l-1-j)\tau^{2l-2-j}\int^{a+1}_a s^{n+j}(s-\tau)^2ds.
\end{split}
\end{alignat}
Define
\begin{alignat*}{1}
E(\tau)=&2l\sum_{k=1}^{2}k\,\tau^{\,n+2l-1-k}\int_{a}^{a+1}s^{k-1}(s-\tau)^2ds\\
       &+n\sum_{j=0}^{2}(2l-1-j)\,\tau^{\,2l-2-j}\int_{a}^{a+1}s^{\,n+j}(s-\tau)^2ds
\end{alignat*}
and
\begin{alignat*}{1}
H(s)=&2l\tau^{\,n+2l-2}+4l\tau^{\,n+2l-3}s
     +n(2l-1)\tau^{\,2l-2}s^{\,n}\\
     &+n(2l-2)\tau^{\,2l-3}s^{\,n+1}
     +n(2l-3)\tau^{\,2l-4}s^{\,n+2}.
\end{alignat*}
Since \(l\ge2\), all coefficients in \(H(s)\) are non-negative. The monomials
\(s^0\), \(s^1\), \(s^n\), \(s^{n+1}\), \(s^{n+2}\) are convex functions on \(s>0\) (their second derivatives are non-negative),
hence their non-negative linear combination \(H(s)\) is convex. Applying Jensen's inequality with the non-negative weight \((s-\tau)^2\) gives
\[
\frac{1}{I_0(\tau)}\int_{a}^{a+1}(s-\tau)^2H(s)\,ds
\;\ge\; H\!\left(\frac{\int_{a}^{a+1}s(s-\tau)^2ds}{\int_{a}^{a+1}(s-\tau)^2ds}\right)
      =H(\bar{s}_\tau),
\]
which is precisely
\begin{equation}
E(\tau)\ge I_0(\tau)\,H(\bar{s}_\tau).
\end{equation}
By using the similar discussion in the proof of Corollary \ref{KC111} and the fact that
\begin{alignat*}{1}
H'(s)=&4l\tau^{\,n+2l-3}
      +n(2l-1)\tau^{\,2l-2}\,n\,s^{\,n-1}\\
      &+n(2l-2)\tau^{\,2l-3}(n+1)s^{\,n}
      +n(2l-3)\tau^{\,2l-4}(n+2)s^{\,n+1}
      \\&>0,
\end{alignat*}
we get
\begin{equation*}
E(\tau)>I_0(\tau)\,H\!\Bigl(\frac{19}{20}\tau\Bigr)\geq \frac{1}{12}H\Bigl(\frac{1}{5}\tau\Bigr)=\frac{(l+nl-n)}{2}Q\Bigl(\frac{1}{5}\Bigr)\tau^{2l+n-1},
\end{equation*}
where
$$Q(s)=2l +4l s+n(2l-1) s^{\,n}+n(2l-2) s^{n+1}+n(2l-3)s^{n+2}.$$

Putting this lower bound of $E(\tau)$ into (\ref{KEOG2}), we get
\begin{alignat}{1}\label{KEOG3}
\begin{split}
n&(2l+n)B-(2l+n)(nA)\tau^{2l}+2l\tau^{2l+n}\geq Q\Bigl(\frac{1}{5}\Bigr)\tau^{n+2l-2}.
\end{split}
\end{alignat}
Using the similar discussion in \cite{JX2}, we get
\begin{alignat}{1} \label{crl1111}
\begin{split}
\sum_{i=1}^k \Lambda_i
\geq &\frac{n}{n+2l}\omega_n\bar{\phi}(0)^{-\frac{2l}{n}} (nA)^{\frac{2l+n}{n}}
+\bar{C}_n\frac{\omega_n\bar{\phi}(0)^{\frac{2n+2-2l}{n}}}{(2l+n)\rho^{2}} (nA)^{\frac{2l+n-2}{n}}
\end{split}
\end{alignat}
where
$$\bar{C}_n\leq Q\Bigl(\frac{1}{5}\Bigr). $$
\textbf{Case I.} If $n+1-l\leq 0$, then the right hand of (\ref{crl1111}) is a decreasing function of $\bar{\phi}(0)$. In this case, we get
\begin{alignat*}{1}
\begin{split}
\sum_{i=1}^k \Lambda_i
\geq &\frac{n}{n+2l}\omega_n\alpha^{-\frac{2l}{n}}\left( \frac{k}{\omega_n}\right)^{\frac{2l+n}{n}}
+\frac{Q(1/5)\omega_n\alpha^{\frac{2n+2-2l}{n}}}{(2l+n)\rho^{2}} \left( \frac{k}{\omega_n}\right)^{\frac{2l+n-2}{n}}.
\end{split}
\end{alignat*}
\textbf{Case II.} If $n+1-l> 0$, then using the similar discussion in \cite{M}, we get
$$\bar{C}_n= \min\Big\{ \frac{2nl}{n+1-l}k^{\frac{2}{n}} ,Q\Bigl(\frac{1}{5}\Bigr) \Big\}.$$
Consequently, we finish our proof.
\cvd

\end{proof}

\begin{Remark}
Since
$$Q\Bigl(\frac{1}{5}\Bigr)\geq \frac{5l}{2},$$
we conclude that Corollary \ref{CCLLL2} improve the following  Theorem 1.2 in \cite{JX2}: if  $2l+n\geq 6$,  then
\begin{alignat*}{1}
\sum_{i=1}^k\Lambda_i \geq&\frac{n}{n+2l}\frac{(2\pi)^{2l}}{(\omega_nV(\Omega))^{\frac{2l}{n}}}k^{\frac{2l+n}{n}}+\frac{5l}{2 (2l+n)\rho^2}\omega_n^{-\frac{2l-2}{n}}\alpha^{\frac{2n-2l+2}{n}}k^{\frac{2l+n-2}{n}}\\
&-\frac{31l\omega_n\alpha^{\frac{3n-2l+3}{n}}}{9(2l+n)\rho^3}\left(\frac{k}{\omega_n} \right)^{\frac{2l+n-3}{n}}
+\frac{5l\omega_n\alpha^{\frac{4n-2l+4}{n}}}{8(2l+n)\rho^4}\left(\frac{k}{\omega_n} \right)^{\frac{2l+n-4}{n}}\\
&+\frac{38l\omega_n\alpha^{\frac{5n-2l+5}{n}}}{25(2l+n)\rho^5}\left(\frac{k}{\omega_n} \right)^{\frac{2l+n-6}{n}}
-\frac{317l\omega_n\alpha^{\frac{6n-2l+6}{n}}}{420(2l+n)\rho^6}\left(\frac{k}{\omega_n} \right)^{\frac{2l+n-6}{n}}.
\end{alignat*}
\end{Remark}

\section{Statements}
\textbf{Conflict of interest} On behalf of all authors, the corresponding author states that there is no conflict of interest.
\\
\\
\textbf{Data} The manuscript has no associated data.

\section{Funding Declaration}

This work was supported  by the National Natural Science Foundation of China (Grant Nos. 12501068, 12271069) and the Natural Science Foundation of Zhejiang Province (Grant No. LQN25A010001).

\vspace{0.6cm}

\section{Acknowledgments}
 Ji would like to thank Professor Kefeng Liu and Professor Hongwei Xu for their constant supports, advices and encouragements.




\end{document}